\documentclass[11pt]{article}

\usepackage{amsmath}
\usepackage{amssymb}
\usepackage{bbm}
\usepackage{amsfonts}
\usepackage{amsthm}
\usepackage{parskip}

\newtheorem{thm}{Theorem}
\newtheorem{prop}{Proposition}
\newtheorem*{CCW}{Carbery-Christ-Wright Uniform Sublevel Set Theorem \cite{carbery1999multidimensional}}
\newtheorem*{CCW2}{Proposition (Carbery-Christ-Wright) \cite{carbery1999multidimensional}}
\newtheorem*{stei}{Theorem (Steinerberger) \cite{steinerberger2019sublevel}}

\title{On a new type of Inequality related to the Uniform Sublevel Set Problem}
\author{John Green}
\date{}

\begin{document}
\maketitle
\begin{abstract}
Recently, Steinerberger \cite{steinerberger2019sublevel} proved a uniform inequality for the Laplacian serving as a counterpoint to the standard uniform sublevel set inequality which is known to fail for the Laplacian. In this note, we give an elementary proof of this result which highlights a step allowing for adaptations to other situations, for instance, we show that the inequality also holds for the heat operator. We formulate some naturally arising questions.
\end{abstract}
\begin{footnotesize}
\end{footnotesize}
\section{Introduction}
A central problem throughout analysis is to understand how oscillatory integrals
$$I(\lambda)=\int e^{i\lambda u(x)}\,dx$$
decay for large values of a real ``frequency parameter" $\lambda$, where $u$ is a real-valued ``phase" function. In general, considerations such as the domain of integration or other functions multiplying the oscillatory factor inside the integral are important, but for the sake of this discussion we will not go into these specifics. Typically this decay will be expressed as
$$|I(\lambda)|\leq C\lambda^{-\delta}$$
for some $\delta>0$. Here we have in mind the idea that as $\lambda$ increases, the small differences in $u(x)$ from moving in $x$ become large differences in $\lambda u(x)$, which in turn corresponds to rapid oscillation in $e^{i\lambda u(x)}$. Thus in the integral, we expect $I(\lambda)$ to decay for large $\lambda$ provided $u$ does not stay near any particular value, and the more quickly $u$ ``moves around", the greater the cancellation we expect to occur, and hence the greater we can take $\delta$ to be. Thus one of the most natural conditions to impose is that $Du$ be bounded below by some positive constant, for some differential operator $D$.

Crucial to many applications and key to the discussion in this paper is the idea of uniformity of the constant $C$ within a large class of phases. A natural example appears when studying the Fourier transform of some density on a $k$-dimensional surface $S$ in $\mathbb{R}^n$. After performing a change of variables, we will have integrals containing an oscillatory factor $e^{-i\phi(x)\cdot\xi}$, where $\phi$ is a function on a piece of $\mathbb{R}^k$ parametrising a piece of $S$, and $\xi$ is the Fourier variable. Writing $\xi=|\xi|\omega$ for $\omega\in S^{n-1}$, we consider $|\xi|$ to be our frequency parameter and $-\phi(x)\cdot\omega$ is a class of phase functions indexed by $\omega$. If we are to obtain estimates on the Fourier transform of the form $C|\xi|^{-\delta}$, we need to make sure the constant $C$ does not blow-up as we vary over $\omega$. In line with the intuition expressed above, the decay of this Fourier transform is well-known to relate to the curvature of the surface, see Stein \cite{stein1993harmonic} for a discussion of the fundamental results on oscillatory integrals and their relation to the Fourier transform of surface measures.

A related problem is the sublevel set problem: given a real-valued function $u$ and a constant $c$ what conditions should we impose so that estimates of the form $|\{x\in\Omega:|u(x)-c|\leq\alpha\}|\leq C\alpha^\delta$ hold for appropriate $\Omega$. It is typical to seek estimates independent of $c$ so that the problem is invariant under shifting $u$ by a constant, and we can assume without loss of generality that $c=0$.

That this should be related is apparent from the intuition expressed above, that oscillatory integrals should observe greater cancellation if $u$ does not spend too much time near a given value. And just as in the oscillatory integral case, we are often not only interested in the best possible $\delta$, but also in the uniformity of the constant $C$ in a class of functions. As mentioned above, typically the class of functions for which we seek uniform bounds is those having $Du$ bounded below by some positive constant, where $D$ is a differential operator. For differential operators where $u$ itself does not appear explicitly, in particular for linear differential operators, this condition is invariant under translation of $u$ by a constant, so uniform estimates are necessarily independent of $c$.

We now recall some discussion from the paper of Carbery-Christ-Wright \cite{carbery1999multidimensional}. Oscillatory integral estimates of the form above are known to imply the corresponsing sublevel set estimates. In the case of monomial derivatives, that is, the differential operators $D^\beta=\partial^{\beta_1}_1\dots\partial^{\beta_n}_n$ for $\beta=(\beta_1,\dots,\beta_n)\in\mathbb{N}_0^n$, it is known that we can take $\delta=(|\beta_1|+\dots+|\beta_n|)^{-1}$ in the sublevel set problem when $D^\beta u$ is bounded below by a positive constant, and that this is the optimal $\delta$ provided no extra conditions are imposed. The same $\delta$ works in the oscillatory integral problem, provided some slightly more restrictive conditions are also imposed, we shall not discuss these here.

However, in all but dimension $1$, this $\delta$ has not been shown to hold with a uniform constant. The main results of the Carbery-Christ-Wright paper are the following, and an analogue for oscillatory integrals with the slightly more restrictive conditions imposed.

\vspace{0.3cm}

\begin{CCW}
Denote the unit cube in $\mathbb{R}^n$ by $Q_n=[0,1]^n$. There exists $C, \delta>0$ such that for any smooth $u$ in a neighbourhood of $Q_n$ having $D^\beta u\geq 1$ on $Q_n$, we have the sublevel set estimates
$$|\{x\in Q_n:|u(x)|\leq\varepsilon\}|\leq C\varepsilon^\delta.$$
Note that $C$ and $\delta$ do not depend on $u$.
\end{CCW}
They also observe that their arguments make sense when $Q_n$ is replaced with different convex sets. Note that this result says nothing about the optimality of $\delta$. It remains open in higher dimensions as to what is the best $\delta$ for which such bounds hold with a uniform constant.

In higher dimensions, we have access to many interesting differential operators, one natural example being the Laplacian. For the Laplacian, one can obtain estimates with $\delta=1/2$ and a non-uniform constant depending on the derivatives of the function up to third order, but in the paper of Carbery-Christ-Wright, it is shown that no uniform estimate can hold for any positive $\delta$.

\vspace{0.3cm}

\begin{CCW2}
For each $\varepsilon\in(0,1/2)$ there exists a smooth $u$ with $\Delta u \equiv 1$ on $[0,1]^2$ but also satisfying the estimate $|\{x\in[0,1]^2:|u(x)|\leq\varepsilon\}|\geq 1-\varepsilon$.
\end{CCW2}
This result extends to higher dimensions by considering this family of counterexamples, and extending them as functions in higher dimensions by asking that they remain constant in the additional variables. Thus no uniform sublevel set estimate holds for any operator consisting of the Laplacian in $2$ or more of the variables plus additional terms - in particular, we observe failure for the heat and wave operators in $2$ or more spatial dimensions.

The above result for the Laplacian is complimented by the result of Steinerberger \cite{steinerberger2019sublevel} which we shall discuss in this paper. It states:

\vspace{0.3cm}

\begin{stei}
There exists a constant $c_n>0$ depending only on the dimension so that if $u:B\rightarrow\mathbb{R}$ satisfies $\Delta u\geq 1$ in $B$, where $B$ is a unit Euclidean ball in $\mathbb{R}^n$, then
$$\|u\|_{L^{\infty}(B)}\cdot|\{x\in B:|u(x)|\geq c_n\}|\geq c_n.$$
\end{stei}
This result is in essence saying that if a sublevel set for some small $\varepsilon$ is large, so that its complement is small, then $\|u\|_{L^{\infty}(B)}$ must be large. Applying this to the above family of examples, we see that $u$ must be very large somewhere on the complement of that sublevel set. The intuition for this can be seen from a basic fact which shall be central to our proof - considering the averages of a function $u$ over balls as a function of the radius, we find that the derivative can be quantified exactly in terms of the Laplacian of $u$, indicating that functions with large Laplacian should have ``large" variations. The interesting aspect of this theorem is the quantification of this fact in a uniform way.

It is worth remarking that just as in the Carbery-Christ-Wright Theorem, the assumptions and hence the conclusions of the statement are invariant under replacing $u$ by $u-c$ for any real number $c$.

Steinerberger posed the basic question of whether we can replace $\|u\|_{L^{\infty}(B)}$ with some power of an $L^p$ norm. An affirmative answer to this question is given in the following theorem.

\vspace{0.3cm}

\begin{thm}\label{LapThm}
Given an open, bounded $\Omega\subseteq\mathbb{R}^n$, there exists a constant $c>0$ depending only on $n$ and $\Omega$ so that if $u:\Omega\rightarrow\mathbb{R}$ satisfies $\Delta u\geq 1$ on $\Omega$, then
$$\|u\|_{L^p(\Omega)}\cdot|\{x\in \Omega:|u(x)|\geq c\}|^{1/p'}\geq c$$
for each $1\leq p\leq \infty$, and $p'$ the conjugate exponent. Note that $c$ does not depend on $p$.
\end{thm}

This result is less interesting in the case $p=1$, since then $p'=\infty$ and this is just giving a lower bound on the $L^1$ norm, which follows more directly from a basic property of subharmonic functions detailed in our proof.

We shall also show how the proof can be extended to work for the heat operator, for which uniform sublevel estimates (including in the case of one spatial variable, as we will examine later) fail.

\vspace{0.3cm}

\begin{thm}\label{HeatThm}
Let $\Omega\subseteq\mathbb{R}^{n+1}$ be open and bounded. We shall denote points in $\mathbb{R}^{n+1}$ by $(x,t)\in\mathbb{R}^n\times\mathbb{R}$ and use $\Delta_x$ for the Laplacian in the first $n$ components. The usual heat operator will be denoted $H=\Delta_x-\partial_t$. Then for those $u$ satisfying $Hu\geq 1$ on $\Omega$ we have the estimate
$$\|u\|_{L^p(\Omega)}\cdot|\{x\in \Omega:|u(x)|\geq c\}|^{1/p'}\geq c$$
for each $1\leq p\leq \infty$, and $p'$ the conjugate exponent, where $c>0$ depends only on $n$ and $\Omega$ and not on $p$ or $u$.
\end{thm}

To avoid having to repeatedly explain the type of inequality to which we refer, we shall describe any inequality having the form
$$\|u\|_{L^p(\Omega)}\cdot|\{x\in \Omega:|u(x)|\geq c\}|^{1/p'}\geq c$$
over some class of functions $u$, with $c>0$ independent of $u$, as a Uniformly Balancing Sublevel Inequality. This name is chosen to stress that the $L^p$ norm is balancing the potentially small complement of the sublevel set, and does so in a uniform way. 

These inequalities are, as one would expect, consequences of uniform sublevel set estimates, since we know that each fixed sublevel set is uniformly small, and hence its complement is uniformly large, which also gives a uniform largeness of the $L^p$ norm. We discuss this observation in the following section.

As mentioned above, a central aspect of the proof of Theorem \ref{LapThm} is a derivative formula for a parameterised family of averages that is expressed in terms of the relevant differential operator - the Laplacian. This will also be the case for Theorem \ref{HeatThm}, however, the analogous formula has an issue that we will need to resolve, which will require us to use some slightly non-standard modifications of the averages considered. 

The paper will be laid out as follows. Section \ref{disc} follows on from the discussion of the introduction, providing a more in-depth discussion of how these inequalities relate to the uniform sublevel set problem and some other properties that do not fit neatly into this introduction. Section \ref{proofs} contains the proofs of the main results, as well as some other results not due to the author, but reformulated in a context which stresses the ways it will be helpful for us, and perhaps that it may be helpful elsewhere. Section \ref{ques} provides a little further discussion of the inequalities that must be postponed until after the proofs, followed by some comments on some naturally-arising questions.

In particular, for which differential operators does a uniformly balancing sublevel inequality hold?

\section{Properties of the inequalities}\label{disc}
Foremost, we confirm that uniformly balancing sublevel inequalities follow from uniform sublevel set estimates. This observation creates a strict hierarchy of problems - uniform oscillatory integral estimates imply uniform sublevel set estimates which imply uniformly balancing sublevel inequalities. Thus further results on these inequalities could provide some insight into the uniform sublevel set problem.

\vspace{0.3cm}

\begin{prop}\label{hier}
Suppose that $\Omega$ is open and bounded and we have the sublevel set estimates $|\{x\in\Omega:|u(x)|\leq\varepsilon\}|\leq C\varepsilon^\delta$ for certain functions $u$ on $\Omega$. Then we have
$$\|u\|_{L^p(\Omega)}\cdot|\{x\in \Omega:|u(x)|\geq c\}|^{1/p'}\geq c$$
for each $1\leq p\leq \infty$, where $p'$ is the conjugate exponent, and $c$ depends only on $C$, $\Omega$ and $\delta$.
\end{prop}
\begin{proof}
By Chebyshev's inequality, we have for each $1\leq p<\infty$ and $\varepsilon>0$
$$\varepsilon|\{x\in\Omega:|u(x)|\geq\varepsilon\}|^{1/p}\leq\|u\|_{L^p(\Omega)}$$
and thus we obtain
$$\|u\|_{L^p(\Omega)}\cdot|\{x\in\Omega:|u(x)|\geq \varepsilon\}|^{1/p'}\geq\varepsilon|\{x\in\Omega:|u(x)|\geq \varepsilon\}|.$$
However, the set on the right hand side is the complement of the sublevel set, and so its measure can be bounded below by $|\Omega|-C\varepsilon^\delta$. Thus we can choose $\varepsilon$ depending only on $C, \Omega$ and $\delta$ so that this is in turn bounded below by $|\Omega|/2$, obtaining
$$\|u\|_{L^p(\Omega)}\cdot|\{x\in\Omega:|u(x)|\geq \varepsilon\}|^{1/p'}\geq\varepsilon|\Omega|/2.$$
Now if $|\Omega|/2\geq 1$, we can bound this below by $\varepsilon$ and take $c=\varepsilon$, else note that replacing $\varepsilon$ by $\varepsilon|\Omega|/2$ on the left-hand side only makes it bigger, in which case we can take $c=\varepsilon|\Omega|/2$.

In either case $c$ depended only on $C$, $\Omega$ and $\delta$. We can take the limit $p\rightarrow\infty$ or give a slightly modified direct argument to obtain the $p=\infty$ result.
\end{proof}
Primarily, we are interested in the problem of whether a uniformly balancing sublevel inequality holds in the set of functions with $Du\geq A>0$, where $D$ is a differential operator, typically linear (for non-linear operators we may consider $Du\leq A<0$ separately, since we cannot simply replace $u$ with $-u$ in this case). At present, we do not know any linear differential operators such that this fails, but we do have a non-linear counterexample - we shall consider $Du=-\det \text{Hess }u$.

Under additional assumptions, such as assuming $u$ to be convex, there are uniform sublevel set estimates associated to the determinant of the Hessian, see Carbery \cite{carbery2010uniform}. We use here an example from the paper of Gressman \cite{gressman2011uniform}, which considers some uniform sublevel set estimates and gives remarks on situations when the uniformity fails. We see that the example given there also gives failure of a uniformly balancing sublevel inequality, in a rather striking way.

Concretely, we consider the operator $Du=(\partial^2_{xy}u)^2-(\partial^2_{xx}u)(\partial^2_{yy}u)$ on $[0,1]^2$, and the family of functions $u_N(x,y)=N^{-1}e^x\sin(Ny)$. Clearly we have that $Du_N=e^{2x}\geq 1$, but given any $c$, we can always take $N$ large enough that $\{x\in[0,1]^2:|u_N(x)|\geq c\}$ is empty, so no uniformly balancing sublevel inequality holds.

We next provide some remarks on the formulation of the uniformly balancing sublevel inequalities. With regards to the power of $1/p'$, we remark that this may not be the largest power possible for a given class of functions and a choice of $p$, for instance in our main theorems. Certainly, in the proof of Proposition \ref{hier} there is no optimal power and it is clear how we could obtain any larger power. It may also be the case that the best power we can obtain in some situations may be less than $1/p'$. Such variant inequalities would still capture the essential feature of an $L^p$ norm being used to ``balance" a product with a superlevel set measure that may be small.

There are, however, certain natural aspects to this choice, namely that we are then considering the product of an $L^p$ norm with an $L^{p'}$ norm, which naturally arises in H\"older's inequality, and consequently in the duality of the $L^p$ spaces. Though not a necessity for it, this natural choice leads to us being able to bound the product below by a constant independent of $p$ in Proposition \ref{hier}, in some discussion below regarding the effect of change of variables, and in the main theorems. This is by no means evidence to suggest that $1/p'$ is the correct power to use, but it is certainly a convenient one.

We might also be interested in considering the superlevel set at a height different from the constant we use as a lower bound. Thus more generally, we might consider inequalities of the form
$$\|u\|_{L^p(\Omega)}\cdot|\{x\in \Omega:|u(x)|\geq c_1\}|^{1/p'}\geq c_2$$
with $c_1,c_2>0$.

Such a consideration would only really be of importance if we were concerned with optimality, since as observed in the proof of Proposition \ref{hier}, replacing $c_1$ and $c_2$ by smaller values still yields a true inequality, so we can certainly replace $c_1$ and $c_2$ with their minimum to assume they are equal. We will not concern ourselves with optimality in the main proofs below, but stress our methods do not merely show existence of a suitable constant, but also allow us to find them constructively - even though there is no reason to expect they are anything close to optimal. 

One other convenient property of uniformity is scaling - for instance, if in the Carbery-Christ-Wright Theorem we instead consider $Du\geq A>0$, we can apply results in the class where $Du\geq 1$ by considering $u/A$ in place of $u$, and it is clear how the constant appearing in the sublevel set bounds must scale. Uniformly balancing sublevel inequalities have the same property - indeed, if such an inequality is known to apply to $u$ in a given class of functions, then whenever $v=Au$, we can apply the inequality to $v/A$ and rearrange the result to obtain
$$\|u\|_{L^p(\Omega)}\cdot|\{x\in\Omega:|u(x)|\geq cA\}|^{1/p'}\geq cA.$$
Thus for linear differential operators we can scale the inequality $Du\geq A$ to apply results in the class $Du\geq 1$, and more generally if the operator has some homogeneity so that $D(\lambda u)=\lambda^\alpha Du$, we can make similar statements - for instance, this applies to the determinant of the Hessian.

Furthermore, uniformity allows us to take limits in the inequality to obtain results for ``rough" functions. For instance, by smoothing out a function that only satisfies a differential inequality in a weak/distributional sense, we can apply the inequality first to smooth approximations and then use a standard limiting argument to conclude that it holds for more general functions. We shall not formulate this precisely, but simply note that our results can thus be extended to greater generality if desired.

There are other useful properties shared with uniform sublevel set estimates. The next example we give shows that we can lift these inequalities to higher dimensions, for instance, Theorem \ref{LapThm} also implies that a uniformly balancing sublevel inequality holds for the Laplacian in the first two variables considered as a differential operator on $\mathbb{R}^3$. Concretely, we can say the following.

\vspace{0.3cm}

\begin{prop}
Let $\Omega_1\subseteq\mathbb{R}^{n_1}$ be open and bounded and suppose that within some set of functions $S$ on $\Omega_1$, we have the uniformly balancing sublevel inequality
$$\|u\|_{L^p(\Omega_1)}\cdot|\{x\in\Omega_1:|u(x)|\geq c\}|^{1/p'}\geq c.$$
Let $\Omega_2$ be a bounded open set in $\mathbb{R}^{n_2}$ and suppose $\Omega\subseteq\mathbb{R}^{n_1+n_2}$ contains $\Omega_1\times\Omega_2$. Then for functions $v$ on $\Omega$ such that their restrictions to $\Omega_1\times\{y\}$ lies in $S$ for each $y\in\Omega_2$, we have
$$\|v\|_{L^p(\Omega)}\cdot|\{x\in\Omega:|v(x)|\geq c\}|^{1/p'}\geq c|\Omega_2|.$$
\end{prop}
\begin{proof}
We write $(x,y)$ for an element of $\mathbb{R}^{n_1}\times\mathbb{R}^{n_2}$. Let $v$ be as in the statement, then
\begin{equation}\label{ineq1}
c\leq\|v(\cdot,y)\|_{L^p(\Omega_1)}\cdot|\{x\in\Omega_1:|v(x,y)|\geq c\}|^{1/p'}
\end{equation}
It is clear that
\begin{align*}
\|v\|_{L^p(\Omega_1\times\Omega_2)}&=\|(\|v(\cdot,y)\|_{L^p(\Omega_1)})\|_{L^p(\Omega_2)}\\
|\{(x,y)\in\Omega_1\times\Omega_2\colon|v(x,y)|\geq c\}|^{1/p'}&=\||\{x\in\Omega_1:|v(x,y)|\geq c\}|^{1/p'}\|_{L^{p'}(\Omega_2)}
\end{align*}
hence an application of H\"older's inequality to inequality (\ref{ineq1}) yields the result for $\Omega_1\times\Omega_2$, which in turn gives the result on $\Omega$.
\end{proof}
We note the last step of the proof holds more generally - if a uniformly balancing sublevel inequality holds on some set, then it holds on any larger set.

Moreover, we can also consider the effect of diffeomorphisms on uniformly balancing sublevel inequalities. Concretely, suppose we have a uniformly balancing sublevel inequality for functions $u$ on $\Omega$ and a diffeomorphism $\phi:\Omega\rightarrow\Omega'$. By the change of variables formula we have
$$\int_\Omega|u(x)|^p\,dx=\int_{\Omega'}|u(\phi^{-1}(x'))|^p|\det J\phi^{-1}(x')|\,dx'$$
where $J\phi^{-1}$ is the Jacobian of $\phi^{-1}$. Let $M$ be the supremum of $|\det J\phi^{-1}(x')|$. Then we have
$$\|u\|_{L^p(\Omega)}\leq \|u\circ\phi^{-1}\|_{L^p(\Omega')}M^{1/p}.$$
As the other term in the uniformly balancing sublevel inequality is an $L^{p'}$ norm, the same reasoning gives
\begin{align*}
c&\leq\|u\|_{L^p(\Omega)}\cdot|\{x\in\Omega:|u(x)|\geq c\}|^{1/p'}\\
&\leq M\|u\circ\phi^{-1}\|_{L^p(\Omega')}\cdot|\{x'\in\Omega':|u\circ\phi^{-1}(x')|\geq c\}|^{1/p'}.
\end{align*}
Hence uniformly balancing sublevel inequalities for a class of $u$ on $\Omega$ yield uniformly balancing sublevel inequalities for the class of $u\circ\phi^{-1}$ on $\Omega'$. Note that going between these classes of functions presents no loss when the Jacobian determinant is constant, such as in the case of invertible linear transformations.

An important consequence of this is the following: Suppose we know a uniformly balancing sublevel inequality holds in the class where $Du(x)\geq 1$, for $D$ a differential operator. For clarity say that $D$ is written in terms of ``$x$ coordinates". Then if we express $D$ in terms of $x'$ coordinates, with $\phi$ being the diffeomorphism giving the change of coordinates, we know that in the class where $D(u\circ\phi^{-1})(x')\geq 1$, a uniformly balancing sublevel inequality holds. In short, the class of differential operators for which uniformly balancing sublevel inequalities hold is invariant under change of coordinates.

As an example, let us consider which constant coefficient linear differential operators satisfy uniformly balancing sublevel inequalities. By using invertible linear transformations, one can reduce the cases to study to certain canonical forms. For example, the theory of quadratic forms tells us that to understand the homogeneous second order examples in $\mathbb{R}^n$, we need only understand those having associated polynomials of the form
$$\sum_{i=1}^{m_1}x_i^2-\sum_{j=m_1+1}^{m_2}x_j^2$$
for $m_2\leq n$.

In $\mathbb{R}^2$, this allows us to give a complete picture - in fact, a uniformly balancing sublevel inequality holds in all cases. The quadratic form $x_1^2+x_2^2$ corresponds to the Laplacian, hence follows from Theorem \ref{LapThm}. All the others satisfy the Carbery-Christ Wright Theorem, so in fact a uniform sublevel estimate holds. This is obvious in all cases but $x_1^2-x_2^2$, but this is $(x_1-x_2)(x_1+x_2)$, so setting $x=x_1-x_2$ and $y=x_1+x_2$, we obtain $xy$ via change of coordinates, which is of the correct form.

With regards to the failure of uniform sublevel set estimates, we note that the counterexample for the Laplacian by Carbery-Christ-Wright \cite{carbery1999multidimensional} relies on Mergelyan's Theorem. One can extend this construction to other differential operators without any difficulty provided an analogue of Mergelyan's Theorem holds for that operator. This is a huge request, and although it may be possible, it is much more reasonable to work with Runge's Theorem, for which many generalisations have been considered - in particular, for the heat operator.

Both Runge's Theorem and Mergelyan's Theorem can be found in Rudin's book \cite{rudin1966real}, and the analogue of Runge's Theorem for the heat operator was given by Jones \cite{jones1975approximation}. Runge's Theorem and Mergelyan's Theorem concern holomorphic functions in the plane, but by considering the real and imaginary parts separately can be considered as a result concerning harmonic functions in the plane. We shall state the necessary consequences of these results during the proof of the forthcoming proposition.

We shall establish, as claimed in the introduction, that even with one spatial dimension we have no uniform sublevel estimate for the heat operator. At the same time, we shall re-establish the Carbery-Christ-Wright counterexample, using an alternative proof communicated by James Wright. We stress that this statement is by no means as general as we could make it, in light of other situations where a Runge-type theorem holds - see, for instance, Kalmes \cite{kalmes2019power}.

\vspace{0.3cm}

\begin{prop}
Consider the operators $\Delta=\partial_{xx}^2+\partial_{tt}^2$ and $H=\partial_{xx}^2-\partial_t$ on $[0,1]^2$. For each $\varepsilon>0$, there exists a smooth $u$ with $\Delta u \equiv 1$ on $[0,1]^2$ but also satisfying the estimate $|\{x\in[0,1]^2:|u(x)|\leq\varepsilon\}|\geq 1-\varepsilon$. Furthermore, for each $\varepsilon>0$, there exists a smooth $u$ with $Hu \equiv 1$ on $[0,1]^2$ but also $|\{x\in[0,1]^2:|u(x)|\leq\varepsilon\}|\geq 1-\varepsilon$.
\end{prop}
\begin{proof}
We will prove both statements simultaneously, indicating the differences as they appear.

The Runge theorem for the Laplacian says that for a harmonic function on an open set $U$ containing a compact set $K$ to be uniformly approximated on $K$ by polynomials harmonic on all of $\mathbb{R}^2$, it is enough that the complement of $K$ is connected.

The Runge theorem for the heat operator says that a temperature (a solution to the heat equation) on an open set $U$ containing a compact set $K$ may be uniformly approximated on $K$ by temperatures on all of $\mathbb{R}^{n+1}$ provided that the $t$-slices of the complement of $U$ have no compact component, that is, the sets $\{x\in\mathbb{R}^n:(x,t)\in U^{c}\}$ have no compact component. For us, then, it is clearly sufficient that the $t$-slices of $U$ be intervals.

In each case, our compact set will be a collection of thin rectangles that cover most of the unit square. To be precise, consider some $0<\delta<1/2$ and let $K$ be the union of the disjoint rectangles
$$K:=\bigcup_{i=1}^{\left\lfloor\frac{4-\delta^2}{4\delta+\delta^2}\right\rfloor} \left[\frac{\delta}{4},1-\frac{\delta}{4}\right]\times\left[i\left(\delta+\frac{\delta^2}{4}\right)-\delta,i\left(\delta+\frac{\delta^2}{4}\right)\right]$$
One sees that $\lfloor(4-\delta^2)/(4\delta+\delta^2)\rfloor(\delta+\delta^2/4)\leq 1-(\delta/4)$, and that the rectangles are separated by $\delta^2/4$, so for instance the $\delta^2/16$ neighbourhood of $K$ (say in the supremum norm) is also a disjoint collection of rectangles. This neighbourhood will be our open set $U$.

Consider $v(x,t)=t^2/2$ for the $\Delta$ case, $v(x,t)=-t$ for the $H$ case. Then $\Delta v$, respectively $Hv$, is identically $1$. Now on each of the thin rectangles of $U$, pick the $t$ coordinate of some point in the rectangle - call it $c$ - and define $w_1(x,t)$ on that component to be $c^2/2$ (respectively $-c$). Since $w_1$ is locally constant, it is harmonic (respectively, a temperature) on $U$. It is easily seen that, since the side length along the $t$ axis of each such rectangle is $\delta+(\delta^2/8)$, $w_1$ uniformly approximates $v$ on $U$ to within $\delta+(\delta^2/8)$.

By Runge's theorem for harmonic functions, it is clear that we can approximate $w_1$ uniformly to within $\delta$ on $K$ by a function $w_2$ harmonic on all of $\mathbb{R}^2$. Thus $w_2$ approximates $v$ on $K$ to within $2\delta+(\delta^2/8)$. Hence $u:=v-w_2$ satisfies $\Delta u \equiv 1$ and $|u(x)|\leq 2\delta+(\delta^2/8)$ on $K$. The analogous statement for the $H$ case is true, since the $t$-slices of $U$ are intervals, and so the we can apply the Runge theorem for temperatures. 

It remains to note that $K$ has large measure. Indeed, the measure of $K$ is
\begin{align*}
\delta\left(1-\frac{\delta}{2}\right)\left\lfloor\frac{4-\delta^2}{4\delta+\delta^2}\right\rfloor&\geq\delta\left(1-\frac{\delta}{2}\right)\left(\frac{4-\delta^2}{4\delta+\delta^2}-1\right)\\
&=\frac{\delta^3+3\delta^2-14\delta+8}{8+2\delta}>\frac{8-14\delta}{8+2\delta}=1-\frac{16\delta}{8+2\delta}\\
&>1-2\delta
\end{align*}
Taking $2\delta+(\delta^2/8)\leq\varepsilon$ completes the proof.
\end{proof}
\section{Proofs of the main results}\label{proofs}
Our proofs will rely on certain results on growth rates of appropriate families of averages, both of which are consequences of some elementary calculations, but they are not always explicitly stated in the literature. For the sake of completeness we produce them here in a way that emphasises their applicability. In the following proofs $u$ will denote a smooth function on a bounded open set $\Omega$.
\subsection{Proof of Theorem \ref{LapThm}}
The derivative formula for the Laplacian is entirely routine and well-known. 

\vspace{0.3cm}

\begin{prop}
Consider for $\overline{B_R(x)}\subseteq\Omega$ the function $\phi:[0,R]\rightarrow\mathbb{R}$ given by $\phi(0)=u(x)$ and the average
$$\frac{1}{|B_r|}\int_{B_r(x)}u(y)\,dy$$
for $0<r\leq R$. Clearly $\phi$ is continuous on $[0,R]$, and on the open interval $(0,R)$ we have
\begin{equation}\label{deriv1}
\phi'(r)=\frac{1}{|B_r|}\int_{B_r(x)}\frac{r^2-|x-y|^2}{2r}\Delta u(y)\,dy.
\end{equation}
\end{prop}
Here $B_r(x)$ denotes the Euclidean ball of radius $r$ centred at $x$, and $|B_r|$ is the Lebesgue measure of the ball.
\begin{proof}
Using the change of variables $y=x-r\tilde{y}$ we have
$$\phi(r)=\frac{1}{|B_1|}\int_{B_1(0)}u(x-r\tilde{y})\,d\tilde{y}.$$
Differentiating under the integral and changing variables again we have 
\begin{align*}
\phi'(r)&=\frac{1}{|B_1|}\int_{ B_1(0)}\nabla u(x-r\tilde{y})\cdot (-\tilde{y})\,d\tilde{y}\\
&=\frac{1}{|B_r|}\int_{ B_r(0)}\nabla u(y)\cdot \frac{y-x}{r}\,dy.
\end{align*}
Now using polar coordinates, with $\sigma_s$ being the induced surface measure on the sphere of radius $s$, we have
\begin{align*}
\phi'(r)&=\frac{1}{|B_r|}\int_0^r\int_{\partial B_r(x)}\nabla u(y)\cdot \frac{y-x}{r}\,d\sigma_s(y)\,ds\\
&=\frac{1}{|B_r|}\int_0^r\frac{s}{r}\int_{\partial B_s(x)}\nabla u(y)\cdot \frac{y-x}{s}\,d\sigma_s(y)\,ds\\
&=\frac{1}{|B_r|}\int_0^r\frac{s}{r}\int_{B_s(x)}\Delta u(y)\,dy\,ds
\end{align*}
where in the last step we used Gauss' Divergence Theorem. We again apply polar coordinates to the inner integral and apply Fubini's Theorem.
\begin{align*}
\phi'(r)&=\frac{1}{|B_r|}\int_0^r\frac{s}{r}\int_0^s\int_{\partial B_t(x)}\Delta u(y)\,d\sigma_t(y)\,dt\,ds\\
&=\frac{1}{|B_r|}\int_0^r\int_t^r\frac{s}{r}\int_{\partial B_t(x)}\Delta u(y)\,d\sigma_t(y)\,ds\,dt\\
&=\frac{1}{|B_r|}\int_0^r\int_{\partial B_t(x)}\frac{r^2-t^2}{2r}\Delta u(y)\,d\sigma_t(y)\,dt.
\end{align*}
Noting $t=|x-y|$ on $\partial B_t(x)$ completes the calculation.
\end{proof}
\begin{proof}[Proof of Theorem \ref{LapThm}]
We shall proceed by contradiction. Suppose that given a positive $c$ we can find $u$ with $\Delta u\geq 1$ so that
\begin{equation}\label{ineq2}
\|u\|_{L^p(\Omega)}\cdot|\{x\in\Omega:|u(x)|\geq c\}|^{1/p'}\leq c
\end{equation}
holds. For $c$ sufficiently small (we shall determine precisely how during the proof), we will obtain a contradiction.

Let $\Omega_\delta$ denote the set of those $x\in\Omega$ that are at least some fixed distance $\delta$ from the boundary. For each such $x$ we can apply mean value inequality for averages over balls $B_{\delta/2}(x)$.

The key observation is that inequality (\ref{ineq2}) provides control on the ``bad part" of each of the averages. In particular, denoting $\{x:|u(x)|\leq c\}$ by $U_c$, we have
\begin{align*}
\frac{1}{|B_{\delta/2}|}\int_{B_{\delta/2}(x)}u(y)dy&=\frac{1}{|B_{\delta/2}|}\left(\int_{B_{\delta/2}(x)\cap U_c}u(y)dy+\int_{B_{\delta/2}(x)\cap (U_c)^c}u(y)dy\right)\\
&\leq (1+|B_{\delta/2}|^{-1})c
\end{align*}
where we have estimated the second integral by H\"older's inequality and the assumed inequality.

Consider now the derivative formula (\ref{deriv1}). We shall denote by $|\partial B_s|=|\partial B_1|s^{n-1}$ the total surface measure of the sphere of radius $s$, and using the assumption $\Delta u\geq 1$, we have that the derivative is bounded below by
\begin{align*}
\frac{1}{|B_r|}\int_0^r|\partial B_1|s^{n-1}\frac{r^2-s^2}{2r}\,ds&=\frac{|\partial B_1|}{|B_r|}\left(\frac{r^{n+2}}{2rn}-\frac{r^{n+2}}{2r(n+2)}\right)\\
&=\frac{|\partial B_1|}{|B_1|}\left(\frac{1}{2n}-\frac{1}{2(n+2)}\right)r=:C_nr.
\end{align*}
Now whenever $x\in\Omega_\delta$, we have by the fundamental theorem of calculus that $u(x)=\phi(0)=\phi(\delta/2)-\int_0^{\delta/2}\phi'(r)\,dr$, which along with $\Delta u\geq 1$ gives
\begin{align*}
u(x)&=\frac{1}{|B_\delta/2|}\int_{B_{\delta/2}(x)}u(y)\,dy-\int_0^{\delta/2}\frac{1}{|B_r|}\int_{B_r(x)}\frac{r^2-|x-y|^2}{2r}\Delta u(y)\,dy\,dr\\
&\leq (1+|B_{\delta/2}|^{-1})c-\int_0^{\delta/2}C_nr\leq (1+|B_{\delta/2}|^{-1})c-(C_n/8)\delta^2.
\end{align*}
It follows that whenever
$$c\leq (1+|B_{\delta/2}|^{-1})^{-1}\frac{C_n\delta^2}{16}$$
we have that $u(x)\leq -(C_n/16)\delta^2$ for each $x\in\Omega_\delta$. Note in particular that $c\leq(C_n/16)\delta^2$.

Now we have that
$$\|u\|_{L^p(\Omega)}\geq\|u\|_{L^p(\Omega_\delta)}\geq (C_n/16)\delta^2|\Omega_\delta|^{1/p}$$
and therefore
\begin{align*}
&\quad\,\,\|u\|_{L^p(\Omega)}\cdot|\{x\in\Omega:|u(x)|\geq c\}|^{1/p'}\\
&\geq\|u\|_{L^p(\Omega)}\cdot|\{x\in\Omega:|u(x)|\geq(C_n/16)\delta^2\}|^{1/p'}\\
&\geq\|u\|_{L^p(\Omega)}\cdot|\Omega_\delta|^{1/p'}\\
&\geq(C_n/16)\delta^2|\Omega_\delta|^{1/p}\cdot|\Omega_\delta|^{1/p'}\\
&=(C_n/16)\delta^2|\Omega_\delta|.
\end{align*}
Thus if we know further that $c<(C_n/16)\delta^2|\Omega_\delta|$, we would arrive at a contradiction. So if we choose $\delta$ so that $\Omega_\delta$ has positive measure, the desired inequality holds when $c<\min\{(C_n/16)\delta^2|\Omega_\delta|,(1+|B_{\delta/2}|^{-1})^{-1}(C_n/16)\delta^2\}$. It is clear that this choice is independent of $p$.
\end{proof}

\subsection{Proof of Theorem \ref{HeatThm}}
For the heat operator case, we shall first define some notation. Let $\Phi$ be the standard heat kernel
$$\Phi(x,t):=\frac{1}{(4\pi t)^{n/2}}e^{-|x|^2/4t}.$$
The heatball centred at $(x,t)$ of radius $r>0$ is the compact set
$$E(x,t;r):=\{(x,t)\}\cup\{(y,s)\in\mathbb{R}^{n+1}:s\leq t,\,\Phi(t-s,x-y)\geq 1/r^n\}.$$
Note that except for at $(x,t)$, $\Phi(t-s,x-y)=1/r^n$ on the boundary.

For convenience we shall also use $E(r)$ to denote $\{(x,t):\Phi(x,t)\geq 1/r^n\}$. Notice that $|E(x,t;r)|=|E(r)|=r^{n+2}|E(1)|$, a fact that follows from the parabolic scaling $(y,s)\mapsto (ry,r^2s)$ taking $E(1)$ to $E(r)$.

Whenever the heatball $E(x,t;R)$ is in $\Omega$, we shall consider the functions $\phi:[0,R]\rightarrow \mathbb{R}$ given by $\phi(0)=u(x,t)$ and
$$\phi(r)=\frac{1}{4r^n}\int_{E(x,t;r)}u(y,s)\frac{|x-y|^2}{(t-s)^2}\,dy\,ds$$
for $0<r\leq R$. These averages were considered in a paper of Watson \cite{watson1973theory}, in which a theory of subtemperatures, analogous to the theory of subharmonic functions, was developed.

In the paper of Watson \cite{watson1973theory}, it is seen that
$$\frac{1}{4r^n}\int_{E(x,t;r)}\frac{|x-y|^2}{(t-s)^2}\,dy\,ds=1$$
which is where the normalisation in the above definition comes from. However, as the precise normalisation will not be important for us, and in order to give a self-contained treatment in this paper, the reader may wish instead to think of the factor of $4r^n$ being written as
$$V(r):=\int_{E(r)}\frac{|y|^2}{s^2}\,dy\,ds=r^nV(1).$$
The equality $V(r)=r^nV(1)$ follows from parabolic scaling. That this quantity is finite can be seen quite easily in higher dimensions, and the techniques we will use are suggestive of methods used to account for unboundedness of the kernel of the averages, $|x-y|^2/(V(r)(t-s)^2)$. In fact, we will only need this result in higher dimensions.

By rearranging the formula defining the level set in $E(r)$, we can give the equivalent expression $\{(y,s):0<s\leq r^2/4\pi,|y|\leq(2ns\log(r^2/4\pi s))^{1/2}\}$. The integral becomes
\begin{align*}
V(r)&=\int_0^{r^2/4\pi}\int_{|y|\leq(2ns\log(r^2/4\pi s))^{1/2}}\frac{|y|^2}{s^2}\,dy\,ds\\
&=\int_0^{r^2/4\pi}\int_{|y|\leq 1}\frac{|y|^2}{s^2}(2ns\log(r^2/4\pi s))^{(n+2)/2}\,dy\,ds\\
&=\int_{|y|\leq 1}|y|^2\,dy\,\int_0^{r^2/4\pi}\left(2ns^{1-\frac{4}{n+2}}\log(r^2/4\pi s)\right)^{(n+2)/2}\,ds.
\end{align*}
The $y$ integral is clearly finite and easy to compute. For $n\geq 3$, the integrand in the $s$ integral is bounded, extending continuously to $s=0$ - indeed, it is clear that $s^{1-\frac{4}{n+2}}\log(r^2/4\pi s)$ tends to $0$ as $s\rightarrow 0$.

Now, we have already observed that the kernel for these averages is unbounded, which means that the first step in the proof of Theorem \ref{LapThm} fails here. Nevertheless, we will later consider a modified kernel which is bounded for fixed $r$, and hence the associated measure can be controlled by a multiple of the Lebesgue measure, so that the argument works. The approach will be to add in some extra spatial variables, apply the derivative formula in higher dimensions, but integrate out the extra variables appearing in the kernel to get a new, bounded kernel in lower dimensions, in a manner not unlike the above calculation. 

First, however, we shall give a proof of the derivative formula for heatballs. This proof is due to Evans \cite{evans10}, but it is not explicitly given there, instead being absorbed into the proof of the corresponding mean value formula for the heat equation. 

\vspace{0.3cm}

\begin{prop}
The function $\phi$ defined above is continuous on $[0,R]$, and in the interval $(0,R)$ its derivative is given by
\begin{equation}\label{deriv2}
\phi'(r)=\frac{n}{r^{n+1}}\int_{E(x,t;r)}Hu(y,s)\log(r^n\Phi(x-y,t-s))\,dy\,ds.
\end{equation}
\end{prop}

\begin{proof}
Using the translation and rescaling $y=x-r\tilde{y},s=t-r^2\tilde{s}$, we have
$$\phi(r)=\frac{1}{4}\int_{E(1)}u(x-r\tilde{y},t-r^2\tilde{s})\frac{|\tilde{y}|^2}{\tilde{s}^2}\,d\tilde{y}\,d\tilde{s}.$$
Continuity of $\phi$ is obvious from the smoothness of $u$ in a neighbourhood of $E(x,t;r)$, and noting that the normalisation $V(1)=4$ in these averages is the correct one. We can differentiate under the integral to obtain
$$\phi'(r)=-\frac{1}{4}\int_{E(1)}\left(2rsu_t+\sum_{i=1}^nu_{x_i}y_i\right)\frac{|y|^2}{s^2}\,dy\,ds$$
suppressing the argument $(x-ry,t-r^2s)$ of $u_{x_i}$ and $u_t$. Scaling back we have
$$\phi'(r)=-\frac{1}{4r^{n+1}}\int_{E(r)}\left(2su_t+\sum_{i=1}^nu_{x_i}y_i\right)\frac{|y|^2}{s^2}\,dy\,ds$$
with the argument $(x-y,t-s)$ suppressed. For convenience let us denote $\psi(y,s)=\log(r^n\Phi(y,s))=n\log(r)-(n/2)\log(4\pi s)-|y|^2/4s$. We have $\psi_{x_i}(y,s)=-y_i/2s$ and thus
$$\int_{E(r)}2su_t\frac{|y|^2}{s^2}\,dy\,ds=-4\int_{E(r)}u_t\sum_{i=1}^n\psi_{x_i}(y,s)y_i\,dy\,ds.$$
Noting that $\psi(y,s)=0$ on the boundary of $E(r)$, we can use integration by parts $\psi_{x_i}$ and $y_iu(x-y,t-s)$ to get that this equals
$$4\int_{E(r)}\psi\sum_{i=1}^n(u_t-y_iu_{tx_i})\,dy\,ds=4\int_{E(r)}\psi\left(nu_t-\sum_{i=1}^ny_iu_{tx_i}\right)\,dy\,ds.$$
Focusing on the second term, we have
\begin{align*}
&-4\int_{E(r)}\sum_{i=1}^n\psi(y,s)y_i\frac{\partial}{\partial s}(-u_{x_i}(x-y,t-s))\,dy\,ds\\
=\,&-4\int_{E(r)}\sum_{i=1}^n\psi_t(y,s)u_{x_i}(x-y,t-s)y_i\,dy\,ds\\
=\,&-4\int_{E(r)}\sum_{i=1}^n\left(\frac{|y|^2}{4s^2}-\frac{n}{2s}\right)u_{x_i}(x-y,t-s)y_i\,dy\,ds
\end{align*}
and hence
$$\int_{E(r)}2su_t\frac{|y|^2}{s^2}\,dy\,ds=4\int_{E(r)}\left[nu_t\psi-\sum_{i=1}^n\left(\frac{|y|^2}{4s^2}-\frac{n}{2s}\right)u_{x_i}y_i\right]dy\,ds.$$
All together this gives
\begin{align*}
\phi'(r)&=\frac{-1}{4r^{n+1}}\int_{E(r)}\left[4nu_t\psi-4\sum_{i=1}^n\left(\frac{|y|^2}{4s^2}-\frac{n}{2s}\right)u_{x_i}y_i+\sum_{i=1}^nu_{x_i}y_i\frac{|y|^2}{s^2}\right]dy\,ds\\
&=\frac{-1}{4r^{n+1}}\int_{E(r)}\left[4nu_t\psi-4\sum_{i=1}^n\frac{-n}{2s}u_{x_i}y_i\right]dy\,ds\\
&=\frac{n}{r^{n+1}}\int_{E(r)}\left[\left(\sum_{i=1}^n\frac{-y_i}{2s}u_{x_i}(x-y,t-s)\right)-u_t(x-y,t-s)\psi\right]dy\,ds.
\end{align*}
Since $\psi_{x_i}(y,s)=-y_i/2s$ and $\psi$ is $0$ on the boundary of $E(r)$, an integration by parts in $y_i$ for each term of the sum yields
\begin{align*}
\phi'(r)&=\frac{n}{r^{n+1}}\int_{E(r)}\Delta u(x-y,t-s)-u_t(x-y,t-s)\psi\,dy\,ds\\
&=\frac{n}{r^{n+1}}\int_{E(r)}Hu(x-y,t-s)\psi(y,s)\,dy\,ds\\
&=\frac{n}{r^{n+1}}\int_{E(x,t;r)}Hu(y,s)\log(r^n\Phi(x-y,t-s))\,dy\,ds
\end{align*}
as desired.
\end{proof}

The last ingredient needed to carry out the proof is the averages over the so-called ``modified heatballs". This idea was used by Kuptsov \cite{kuptsov1981mean}, our treatment follows a paper by Watson \cite{watson2002elementary} giving a review of the main ideas and some further results. Conveniently, we won't need anything more than the absolute basics, which we summarise here.

Starting with a function $u$ on an open subset $\Omega$ of $\mathbb{R}^{n+1}$, we examine the above averages $\phi$ for $u$ being considered as a function on $\mathbb{R}^m\times\Omega$, by considering a function $\tilde{u}$ that is independent of the first $m$ variables - that is, for $\xi\in\mathbb{R}^m$, $(x,t)\in\Omega$, set $\tilde{u}(\xi,x,t)=u(x,t)$ and consider the averages $\phi$ for $\tilde{u}$, which clearly take the form
$$\phi(r)=\frac{1}{4r^{m+n}}\int_{E(\xi,x,t;r)}\frac{|x-y|^2+|\xi-\eta|^2}{(t-s)^2}u(y,s)\,d\eta\,dy\,ds.$$
Since $u$ does not dependent on $\eta$, we can carry out the integration in $\eta$. As above, we rearrange the superlevel set formula defining the heatball to observe that $E(\xi,x,t;r)$ is the set of $(y,\eta,s)$ satisfying
$$0\leq t-s\leq r^2/4\pi,|x-y|^2+|\xi-\eta|^2\leq 2(m+n)(t-s)\log(r^2/4\pi(t-s)).$$
The $(n,m)$-modified heatball is the projection of $E(\xi,x,t;r)$ onto the last $n+1$ coordinates, denoted $E_m(x,t;r)$, and by carrying out the integration in $\eta$ we obtain a new kernel on $E_m(x,t;r)$ that will be bounded for large enough $m$, as we shall demonstrate.

For fixed $y$, $s$, we must integrate over $|\xi-\eta|\leq A=A(x-y,t-s)$, given by
$$A(x-y,t-s):=\left(2(t-s)(m+n)\log\left(\frac{r^2}{4\pi(t-s)}\right)-|x-y|^2\right)^{1/2}.$$
We compute the integral in polar coordinates
\begin{align*}
\int_{|\xi-\eta|\leq A}\frac{|\xi-\eta|^2+|x-y|^2}{(t-s)^2}\,d\eta&=|\partial B_1|\int_0^A\frac{r^2+|x-y|^2}{(t-s)^2}r^{m-1}\,dr\\
&=\frac{|\partial B_1|}{(t-s)^2}\left(\frac{A^{m+2}}{m+2}+|x-y|^2\frac{A^m}{m}\right)\\
&=\frac{|B_1|}{(t-s)^2}A^m\left(\frac{mA^2}{m+2}+|x-y|^2\right)
\end{align*}
where we used the basic formula $|\partial B_1|=m|B_1|$. Substituting the value of $A^2$ in the brackets gives
$$\frac{2|B_1|}{m+2}A^m\left(\frac{m(m+n)}{t-s}\log\left(\frac{r^2}{4\pi(t-s)}\right)+\frac{|x-y|^2}{(t-s)^2}\right).$$
Dividing this by the normalisation $1/4r^{m+n}$ gives the kernel $K_r(x-y,t-s)$ (a non-negative density) of the average over the $(n,m)$ modified heatball $E_m(x,t;r)$. That is to say,
$$\phi(r)=\int_{E_m(x,t;r)}K_r(x-y,t-s)u(y,s)\,dy\,ds.$$
For $m\geq 3$, we can show that this kernel is bounded with bounds depending only on $n,m$ and $r$. Indeed, we can immediately replace $x-y$ with $y$ and $t-s$ with $s$ via a change of variables, and consider bounding $K_r(y,s)$ on the set $\{(y,s):\Phi(0,y,s)\geq 1/r^{m+n}\}$ (this is the reflection in the final coordinate of $E_m(0,0;r)$), where here $\Phi$ is the heat kernel in $m+n$ spatial dimensions.

We just need to check that the kernel is bounded near $(0,0)$. Observing that $A^m$ is bounded by $(2s(m+n)\log(r^2/4\pi s))^{m/2}$, we see that since in (the reflection of) the modified heatball we have $|y|^2\leq 2s(m+n)\log(r^2/4\pi s)$, the kernel can be bounded by a constant depending on $n,m$ and $r$ multiplied by
$$s^{(m-2)/2}\log\left(\frac{r^2}{4\pi s}\right)^{(m+2)/2}.$$
Similarly to before, we can easily see that as $s$ goes to $0$, provided $m\geq 3$, this quantity goes to $0$, hence the kernel extends continuously to $(0,0)$ and is thus bounded above on the modified heatball. 
\begin{proof}[Proof of Theorem \ref{HeatThm}]
In view of the above, the proof proceeds almost identically to Theorem \ref{LapThm}, with some small differences that we can easily address.

Suppose that given a positive $c$, we can always find $u$ with $Hu\geq 1$ so that
\begin{equation}\label{ineq3}
\|u\|_{L^p(\Omega)}\cdot|\{x\in\Omega:|u(x)|\geq c\}|^{1/p'}\leq c.
\end{equation}
Fix $m\geq 3$. Let $R$ be fixed so that the set $\Omega_R$, consisting of those points $(x,t)$ for which the modified heatball $E_m(x,t;R)$ is contained in $\Omega$, has positive Lebesgue measure. Let $M_R$ be the maximum value attained by the kernel $K_R$. Then for each $(x,t)\in\Omega_R$, by splitting the integral on the right over the sets where $|u|\leq c$ and $|u|\geq c$, applying H\"older's inequality and the inequality (\ref{ineq3}) to the latter, we obtain
\begin{align*}
\int_{E_m(x,t;R)}K_R(x-y,t-s)u(y,s)\,dy\,ds&\leq M_R\int_{E_m(x,t;R)}|u(y,s)|\,dy\,ds\\
&\leq M_R(|E_m(x,t;R)|+1)c=:C_Rc.
\end{align*}
Note that the extended function $\tilde{u}(\xi,x,t)=u(x,t)$ considered in the proceeding discussion has $H\tilde{u}(\xi,x,t)=Hu(x,t)\geq 1$. By the fundamental theorem of calculus for $\phi$ and the derivative formula (\ref{deriv2}) applied on the $(m+n)$-dimensional heatball, we have
\begin{align*}
u(x,t)&=\tilde{u}(\xi,x,t)=\phi(0)=\phi(R)-\int_0^R\phi'(r)\,dr\\
&=\int_{E_m(x,t;R)}K_R(x-y,t-s)u(y,s)\,dy\,ds\\
&-\int_0^R\frac{1}{4r^{m+n+1}}\int_{E(\xi,x,t;r)}v_{\xi,x,t,r}(\eta,y,s)\,d\eta\,dy\,ds\,dr
\end{align*}
where $v_{\xi,x,t,r}$ is shorthand for $H\tilde{u}(\eta,y,s)\log(r^{m+n}\Phi(\xi-\eta,x-y,t-s))$.
Note that $\log(r^{m+n}\Phi(\xi-\eta,x-y,t-s))$ is positive on the heatball $E(\xi,x,t;r)$, and is at least $m+n$ on the heatball $E(\xi,x,t;r/e)$. Hence we have the bound
$$\int_{E(\xi,x,t;r)}v_{\xi,x,t,r}(\eta,y,s)\,d\eta\,dy\,ds\geq (m+n)|E(\xi,x,t;r/e)|.$$
The right hand side is equal to $(m+n)|E(\xi,x,t;1)|(r/e)^{m+n+2}=C_{m,n}r^{m+n+2}$. Using this bound with the bound $\phi(R)\leq C_Rc$, it follows that for $(x,t)\in\Omega_R$,
$$u(x,t)\leq C_Rc-(C_{m,n}/8)R^2.$$
The remainder of the proof features no major differences to the proof of Theorem \ref{LapThm}. If we take $c\leq C_R^{-1}(C_{m,n}/16)R^2$, then $u(x,t)\leq -(C_{m,n}/16)R^2$. If also $c\leq (C_{m,n}/16)R^2$, then
$$\|u\|_{L^p(\Omega)}\cdot|\{x\in\Omega:|u(x)|\geq c\}|^{1/p'}\geq (C_{m,n}/16)R^2|\Omega_R|$$
and the contradiction follows if we suppose further that $c$ is also less than $(C_{m,n}/16)R^2|\Omega_R|$.
\end{proof}

\section{Remarks and questions}\label{ques}
In light of earlier discussion, it seems of interest to determine the generality in which uniformly balancing sublevel inequalities hold. We pose the following question.

\textbf{Question.} Given a differential operator $D$ on an open set $\Omega$, we are interested in whether there exist constants $c,\alpha>0$ and $1\leq p\leq \infty$ such that an inequality of the form
$$\|u\|_{L^p(\Omega)}\cdot|\{x\in\Omega:|u(x)|\geq c\}|^\alpha\geq c$$
holds whenever $Du\geq A>0$ in $\Omega$. Can we classify which differential operators possess such inequalities? In particular, could it be that a uniformly balancing sublevel inequality holds for all linear differential operators?

In order to motivate the search for a linear counterexample, let us reflect on the above proofs and examine what generality they are likely to extend to. The key element of the proofs was the relation of a rate of change for a parameterised family of averages to the differential operator in question, either by means of a derivative or just a nicely-quantified difference. Such parameterised families of averages occur in the study of mean value formulae for PDE, where often one seeks to establish that such a family of averages is constant if and only if a function satisfies a certain PDE or family of PDEs. In fact, if we allow for averages against measures that are not necessarily positive, a vast number of linear PDE solutions can be characterised this way. Pokrovskii \cite{pokrovskii1998mean} establishes a method for constructing such measures, generalising the ideas of Zalcman \cite{zalcman1973mean}.

We need our averages to be with respect to a measure which can at least be locally bounded by some power of the Lebesgue measure, as is constructed by Pokrovskii \cite{pokrovskii1998mean}, if we are to find straightforward extensions to the above proofs. As we saw in the proofs above, the contrapositive of the statement allowed us to bound averages against the Lebesgue measure over a fixed set in a uniform way, which is why we need to compare to the Lebesgue measure, typically by means of a bounded density as for the heat operator.

The main difficulty seems to be in producing a formula expressing the growth of averages in terms of the differential operator, integrated against a positive measure. This property perhaps holds only for certain naturally arising monotone families of averages, which seems to be related to those differential operators having good theories of subsolutions and supersolutions where we have access to maximum principles, for instance. Indeed, our arguments were based on the idea that for such operators, inequalities of the form $Du\geq 1$ represent a stronger property than that of a usual subsolution, where by comparison with a constant solution we would expect a quantifiably large deviation.

In view of this, it seems reasonable to expect that for sufficiently nice elliptic and parabolic operators, where we indeed have a developed theory for subsolutions and supersolutions, modifications of the above arguments will prove fruitful.

This suggests another interesting avenue to pursue. Consider the wave operator $\partial^2_t-\Delta_x$. In one spatial dimension, the wave operator can be expressed as $\partial^2_t-\partial^2_x=(\partial_t-\partial_x)(\partial_t+\partial_x)$, which by the simple change of co-ordinates $x'=t-x, y'=t+x$ is the same as $\partial_{x'}\partial_{y'}$, and so this in fact satisfies a uniform sublevel set estimate. However, in more than one spatial dimension, simply by considering functions that are constant in $t$, we see that no uniform sublevel set estimate holds due to the failure for the Laplacian. However, the wave operator is neither elliptic nor parabolic, which suggests that the methods of this paper may not be helpful in establishing a uniformly balancing sublevel inequality.

Nevertheless, this is not to suggest that a uniformly balancing sublevel inequality fails to hold for the wave operator, but that the wave operator seems to be worthy of further inspection.

\textit{Acknowledgements.} The author is supported by a UK EPSRC scholarship at the Maxwell Institute Graduate School. The author would like to thank Prof. James Wright for many helpful discussions, along with various suggestions and improvements; in particular the inclusion of the discussion on failure of uniform sublevel set estimates via Runge-type theorems at the end of section \ref{disc}.

\bibliographystyle{abbrv}
\bibliography{Bibliography1}

\begin{thebibliography}{10}

\bibitem{carbery2010uniform}
A.~Carbery.
\newblock A uniform sublevel set estimate.
\newblock {\em Harmonic Analysis and Partial Differential Equations, Contemp.
  Math}, 505:97--103, 2010.

\bibitem{carbery1999multidimensional}
A.~Carbery, M.~Christ, and J.~Wright.
\newblock Multidimensional van der corput and sublevel set estimates.
\newblock {\em Journal of the American Mathematical Society}, 12(4):981--1015,
  1999.

\bibitem{evans10}
L.~C. Evans.
\newblock {\em Partial differential equations}.
\newblock American Mathematical Society, 2010.

\bibitem{gressman2011uniform}
P.~T. Gressman.
\newblock Uniform geometric estimates of sublevel sets.
\newblock {\em Journal d'Analyse Math{\'e}matique}, 115(1):251--272, 2011.

\bibitem{jones1975approximation}
B.~F. Jones.
\newblock An approximation theorem of runge type for the heat equation.
\newblock {\em Proceedings of the American Mathematical Society},
  52(1):289--292, 1975.

\bibitem{kalmes2019power}
T.~Kalmes.
\newblock Power bounded weighted composition operators on function spaces
  defined by local properties.
\newblock {\em Journal of Mathematical Analysis and Applications},
  471(1-2):211--238, 2019.

\bibitem{kuptsov1981mean}
L.~P. Kuptsov.
\newblock Mean property for the heat-conduction equation.
\newblock {\em Mathematical notes of the Academy of Sciences of the USSR},
  29(2):110--116, 1981.

\bibitem{pokrovskii1998mean}
A.~V. Pokrovskii.
\newblock Mean value theorems for solutions of linear partial differential
  equations.
\newblock {\em Mathematical Notes}, 64(2):220--229, 1998.

\bibitem{rudin1966real}
W.~Rudin.
\newblock {\em Real and complex analysis}.
\newblock McGraw-Hill, 1966.

\bibitem{stein1993harmonic}
E.~M. Stein.
\newblock {\em Harmonic analysis: real-variable methods, orthogonality, and
  oscillatory integrals}, volume~3.
\newblock Princeton University Press, 1993.

\bibitem{steinerberger2019sublevel}
S.~Steinerberger.
\newblock On sublevel set estimates and the laplacian.
\newblock {\em Potential Analysis (to appear)}, 2020.

\bibitem{watson1973theory}
N.~A. Watson.
\newblock A theory of subtemperatures in several variables.
\newblock {\em Proceedings of the London Mathematical Society}, 3(3):385--417,
  1973.

\bibitem{watson2002elementary}
N.~A. Watson.
\newblock Elementary proofs of some basic subtemperature theorems.
\newblock In {\em Colloquium Mathematicum}, volume~1, pages 111--140, 2002.

\bibitem{zalcman1973mean}
L.~Zalcman.
\newblock Mean values and differential equations.
\newblock {\em Israel Journal of Mathematics}, 14(4):339--352, 1973.

\end{thebibliography}

John Green,\\ Maxwell Institute of Mathematical Sciences and the School of Mathematics,\\ University of Edinburgh,\\ JCMB, The King’s Buildings,\\ Peter Guthrie Tait Road,\\ Edinburgh, EH9 3FD,\\ Scotland\\ Email: \texttt{J.D.Green@sms.ed.ac.uk}
\end{document}